\documentclass[11pt]{article}
\usepackage{setspace}
\singlespace
\usepackage{amsmath,amsthm,amssymb,color,array,multirow}
\usepackage{epsfig,amsfonts,latexsym, hyperref,  mathrsfs}
\usepackage{natbib}
\usepackage{geometry}
\usepackage{tikz}
\usetikzlibrary{shadows,trees}

\hypersetup{
    colorlinks=true,
    linkcolor=red,
    citecolor=blue,
    urlcolor=blue,
    pdfborder={0 0 0}
}

\newtheorem{theorem}{Theorem}[section]

\newtheorem{lemma}{Lemma}[section]

\theoremstyle{definition}

\numberwithin{equation}{section}

\allowdisplaybreaks

\title{Asymptotic freeness of sample covariance matrices via embedding}
\author{
\hspace{.05\textwidth}
 \parbox[t]{0.4\textwidth}{{\sc Monika Bhattacharjee}
 \thanks{Research  supported by the seed grant from IIT Bombay.}
 \\ {\small Department of Mathematics\\ Indian  Institute of Technology, Bombay\\  monaiidexp.gamma@gmail.com}}
\hspace{0.5 cm}\parbox[t]{0.4\textwidth}{{\sc Arup Bose}
 \thanks{Research  supported by J.C.~Bose National Fellowship, Dept.~of Science and Technology, Govt.~of India.}
 \\ {\small Statistics and Mathematics Unit, Kolkata\\ Indian Statistical Institute\\  bosearu@gmail.com\\}}}

\date{\today}

\begin{document}
\maketitle

\begin{abstract}
We present an alternative proof of asymptotic freeness of independent sample covariance matrices, when the dimension and  the sample size grow 
at the same rate, by embedding these matrices into Wigner matrices of a larger order and using asymptotic freeness of independent Wigner and deterministic matrices.
\end{abstract}

Non-crossing partition; Sample covariance matrix; Semi-circular variable; Wigner matrix. 
\bigskip

\section{Introduction}                                      
Free probability was introduced by \cite{V1991} and soon thereafter, connections were forged between asymptotic freeness of random matrices and free probability by several authors, see for example  \cite{D1993, NP1993,  S1997, V1998} and \cite{HP1999}. 
Asymptotic freeness of many important random matrices have been established and these results have been applied to other domains including high-dimensional statistics (see \cite{BM2018}) and wireless communications (see \cite{CD2011}). 

The most remarkable results in this area are asymptotic freeness of (i) independent Wigner matrices and deterministic matrices  (see \cite{NS2006} and \citep{ZAG2010}),  and (ii) independent sample covariance matrices (see \cite{CC2004}). There is a common thread in the proof of these results. These matrices are viewed as elements of the non-commutative space of random matrices with the state as Etr and the asymptotic freeness is shown by proving that all tracial moments converge and the limits obey the freeness criterion. The arguments and techniques used in the proofs are quite alike and inter-connected. 

Recall that a sample covariance matrix is of the form $XX^*$ where $X$ is a matrix of independent random variables. On the other hand the entries of a Wigner matrix are also independent random variables with the additional restriction of symmetry imposed on the matrix. 
It is thus natural to expect that result (i)  can be leveraged to prove result (ii).
We show that this is indeed possible by connecting the sample covariance matrices with Wigner and deterministic matrices of higher orders.
This is usually known as the embedding technique.
While embedding techniques have been often used in free probability and random matrix theory (see \cite{speicher2009free, belinschi2017analytic}),   
to the best of our knowledge, the presented alternative proof 
does not seem to be directly available in the literature. 

\section{Necessary concepts and results from free probability}
We give a brief introduction to free probability, borrowing heavily 
from \cite{NS2006} and \cite{ZAG2010}.  


\noindent \textbf{Non-crossing partitions and Kreweras complement}. For integer $k\geq 1$, $NC(k)$ is the lattice of non-crossing partitions of $\{1, \ldots , k\}$,  and $\mu (\cdot,  \cdot)$ is the M\"{o}bius function on $NC(k), k \geq 1$. 
%
Consider numbers $\bar{1},\bar{2},\ldots, \bar{n}$ and interlace them with $1,2,\ldots,n$ as:
$1, \bar{1}, 2, \bar{2},\ldots, n, \bar{n}.$
For any $\pi \in NC(n)$, 
its \textit{Kreweras complement} $K(\pi) \in NC(\bar{1},\bar{2},\ldots,\bar{n}) \cong NC(n)$ is the largest element among those $\sigma \in NC(\bar{1},\bar{2},\ldots, \bar{n})$ such that  
$\pi \cup \sigma \in NC(1,\bar{1},2,\bar{2},\ldots, n,\bar{n}).$ 

\noindent \textbf{Non-commutative probability space}.
A pair $(\mathcal{A},\varphi)$ is said to be a \textit{non-commutative $*$-probability space} (NCP), 
if $\mathcal{A}$ is a unital $*$-algebra over the set of complex numbers $\mathbb{C}$, with unity/identity $1_{\mathcal{A}}$, and the \textit{state} $\varphi:\mathcal{A} \rightarrow \mathbb{C}$, is a linear functional such that $\varphi(1_{\mathcal{A}})=1$. 
The elements in $\mathcal{A}$ are called non-commutative \textit{random variables}.
If $a=a^{*}$, then $a$ is called \textit{self-adjoint}. The state $\varphi$ is \textit{tracial} if $\varphi(ab) = \varphi(ba),\ \ \forall a,b \in \mathcal{A}$. All states discussed in this article are tracial.
For any positive integer $n$, let $\mathcal{M}_{n}(\mathbb{C})$ be the set of $n\times n$  random matrices with complex entries with the usual operations of addition, multiplication, and $*$ (conjugate transpose). 
Let $n^{-1}{\rm{ETr}}: \mathcal{M}_{n}(\mathbb{C}) \rightarrow \mathbb{C}$ be the expected normalized trace,
$n^{-1}{\rm{ETr}}(a) = \frac{1}{n}\sum_{i=1}^{n}{\rm E}(\alpha_{ii})\, \ \ \forall a = ((\alpha_{ij}))_{i,j=1}^{n} \in \mathcal{M}_{n}(\mathbb{C}).$ Then $(\mathcal{M}_{n}(\mathbb{C}),n^{-1}{\rm{ETr}})$ is an NCP and the state is tracial. 
\noindent \textbf{Convergence of variables}.
Let $(\mathcal{A}_{n}, \varphi_n)$ be a sequence of NCP. Variables $\{a_i^{(n)},  1\leq i \leq t\}$  from $\mathcal{A}_n$ are said to converge if for every polynomial $\Pi$ in these variables,
$\varphi_{n}\big(\Pi(a_i^{(n)},a_i^{*(n)}:1 \leq i \leq t)\big)$ converges. In that case we can define a polynomial $*$-algebra $\mathcal{A}$ generated by some indeterminate variables $\{a_i,a_i^{*}:1 \leq i \leq t\}$ and define a state $\varphi$ on $\mathcal{A}$ by the above limits.


\noindent \textbf{Free cumulants}.
 Let $(\mathcal{A},\varphi)$ be an NCP. Let $\mathbb{N}$ be the set of all natural numbers. Then $\varphi$ is extended to 
$\mathcal{A}^{k}$ by  
$$\varphi_{k} (a_{1},a_{2},\ldots, a_{k}) := \varphi(a_{1}a_{2}\ldots a_{k}).$$ 
These functionals are used to define \textit{multiplicative functionals} on $NC(k), k \geq 1$ as follows.  
Let $\pi = \{V_{1}, V_{2}, \ldots, V_{r} \} \in NC(k)$, where $\{V_i\}$ are the block of $\pi$. Then
$$\varphi_{\pi}[a_{1},a_{2},\ldots, a_{k}] := \varphi(V_1)[a_{1},a_{2},\ldots,a_{k}]\cdots \varphi(V_r)[a_{1},a_{2},\ldots,a_{k}],$$
where 
$$\varphi(V)[a_{1},a_{2},\ldots, a_{k}] :=\varphi_{s}(a_{i_1},a_{i_2},\ldots,a_{i_s})=\varphi(a_{i_1}a_{i_2}\ldots a_{i_s})$$ if $V = \{i_1,  \ldots ,  i_s, i_1 < i_2 < \ldots < i_s\}$. 

The \textit{free cumulant} functionals on $\mathcal{A}^k$ are defined as 
$$\kappa_{k}(a_{1},a_{2},\ldots, a_{k}) = \sum_{\sigma \in NC(k)} \varphi_{\sigma}[a_{1},a_{2},...,a_{k}]\mu(\sigma, 1_{k}),$$
where  $1_k$ is the single block partition of $\{1,2,\ldots,k\}$.  If $a_1=\cdots =a_k=a$ then we write 
$\kappa_{k}(a_{1},a_{2},\ldots, a_{k})$ as $\kappa_k(a)$. 
The multiplicative extensions of these free cumulants are given by
$$\kappa_{\pi}[a_{1},a_{2},...,a_{k}] := \sum_{\stackrel{\sigma \in NC(k)}{\sigma \leq \pi}} \varphi_{\sigma}[a_{1},a_{2},\ldots, a_{k}]\mu(\sigma, \pi).$$
Note that
$$\varphi_{1_k}[a_1,a_2, \ldots, a_k] = \varphi_{k}(a_1,a_2,\ldots, a_k) = \varphi(a_1a_2\ldots a_k), \ \ 
\kappa_{1_k}[a_1,a_2,\ldots, a_k] = \kappa_{k}(a_1,a_2,\ldots, a_k).$$  
\noindent Moreover, it can be shown that
\begin{eqnarray}
\varphi(a_1a_2\cdots a_k) = \sum_{\pi \in  NC(k)} \kappa_{\pi}[a_1,a_2,\cdots, a_k]. \label{eqn: m1to14.1}
\end{eqnarray}

\noindent \textbf{Free independence}. 
Let $(\mathcal{A}, \varphi)$ be an NCP.  Unital $*$-sub-algebras $(\mathcal{A}_{i})_{i \in I}$ of $\mathcal{A}$  are \textit{freely independent} (strictly speaking $*$-free) if for all $k \geq 2$ and for all $a_{i} \in \mathcal{A}_{i(j)}$,  $j=1,2,\ldots,k$, $i(j)\in I$, 
we have $\kappa_{k}(a_{1},a_{2},\ldots,a_{k}) = 0$ whenever at least two $i(j), j=1, \ldots , k$ are different.
In short, if all \textit{mixed free cumulants} 
are zero. Collections of variables are said to be free if the $*$-algebra generated by the collections are free. It is easy to see that $1_{\mathcal{A}}$ is free of all other variables. 
It is known that two collections $\{a_{i}, 1\leq i \leq k\}$ and $\{b_{i},1\leq i \leq k\}$ 
are free if and only if for all choices of 
$n \geq 1$ and indices $1\leq i(t), j(t)\leq k, t=1, \ldots n$, 
\begin{eqnarray}
\varphi(a_{i(1)}b_{j(1)}a_{i(2)}b_{j(2)}\ldots a_{i(n)}b_{j(n)}) = \hspace{-7pt}\sum_{\pi \in NC(n)}\hspace{-5pt}\kappa_{\pi}[a_{i(1)},a_{i(2)},\ldots,a_{i(n)}]\varphi_{K(\pi)}[b_{j(1)},b_{j(2)},\ldots,b_{j(n)}]. \label{eqn: mp6}
\end{eqnarray}


Collections of variables $\{a_{ij}^{(n)}: 0 \leq i \leq t\}$, are said to be  \textit{asymptotically free} across $1 \leq j \leq k$,  if 
they converge jointly to $\{a_{ij}\}$, $0 \leq i \leq t, 1 \leq j \leq k$ which are free across $1 \leq j \leq k$. 




\noindent \textbf{Asymptotic freeness of Wigner and deterministic matrices}. 
A self-adjoint variable $s$ is called a \textit{standard semi-circular variable} if 
\begin{eqnarray}
\kappa_{k}(s)=0, \ \ {\rm for \ all}\ \ k\neq 2,\ \ {\rm and}\ \ \kappa_2(s)=1. \label{eqn: s6}
\end{eqnarray}
A $p\times p$ real symmetric random matrix is a \textit{Wigner matrix} 
if the entries are independent with mean 0 variance 1 entries on and above the diagonal. 
The following theorem is well known. See Theorems 5.4.2 and 5.4.5 of \cite{ZAG2010} for proof.  

\begin{theorem} \label{lem: wig}
Let $\{W_p^{(i)} = ((w_{jk,p}^{(i)}))_{1 \leq j,k \leq p}:\ 1 \leq i \leq m\}$ be 
independent $p\times p$ Wigner matrices such that $\sup_{p \geq 1} \sup_{1 \leq j,k \leq p} \mathbb{E}|w_{jk,p}^{(i)}|^r < \infty$ for all $r \geq 1$, $1 \leq i \leq m$. Let $\{D^{(i)}_p = ((d_{jk}^{(i)}))_{1 \leq j,k \leq p}:\ 1 \leq i \leq u\}$ be
$p\times p$ deterministic matrices 
which  converge jointly as elements of $(\mathcal{M}_{p}(\mathbb{C}),p^{-1}{\rm{ETr}})$ 
as $p \to \infty$, and $\sup_{p \geq 1}\max_{1\leq j,k\leq p}|d_{jk}^{(i)}| < \infty$ for all $1 \leq i \leq u$. Then the following statements hold.
\begin{enumerate}
\item As $p\to \infty$, $\{p^{-1/2}W_p^{(i)}:\ 1 \leq i \leq m\}$ converge as elements of $(\mathcal{M}_{p}(\mathbb{C}),p^{-1}{\rm{ETr}})$ to $\{s^{(i)}:\ 1 \leq i \leq m\}$  which are free standard semi-circular variables in some NCP $(\mathcal{A}, \varphi)$.


\item  $\{p^{-1/2}W_p^{(i)}:\ 1 \leq i \leq m\}$  and $\{D^{(i)}_p:\ 1 \leq i \leq u\}$ are asymptotically free.
\end{enumerate}
\end{theorem}



\section{Asymptotic freeness of sample covariance matrices}

 A self-adjoint variable $g$ is said to be a \textit{Mar\v{c}enko-Pastur variable} with parameter $y \in (0,\infty)$ if 
\begin{eqnarray}
\kappa_{k}(g) = y^{k-1},\ \forall k \geq 1. \label{eqn: mpcum}
\end{eqnarray} A criterion for freeness of  Mar\v{c}enko-Pastur variables is given in Lemma \ref{lem: mpfreemoment}. 
This will be useful in the proof of asymptotic freeness of independent sample covariance matrices via embedding.  


For any subset $A$ of $\mathbb{N}$,  $NC(A)$ will denote the set of non-crossing partitions of $A$. For any set $B$, $\#B$ denotes the number of elements in $B$. 
Let $\boldsymbol{\tau} = (\tau_1,\tau_2,\ldots,\tau_k) \in \{1,2,\ldots,m\}^k$.  Define
\begin{eqnarray}
J_i(\boldsymbol{\tau}) &=& \{j \in \{1,2,\ldots,k\}:\ \tau_j=i\},\ \ T_i = \# J_i(\boldsymbol{\tau}),\ \ 1\leq i \leq m.\ \label{eqn: Ti}
\end{eqnarray}
Note that 
$\cup_{i=1}^{m}J_i (\boldsymbol{\tau}) = \{1,2,\ldots,k\}\ \ \text{and}\ \ \sum_{i=1}^{m}T_i = k$. Define
\begin{eqnarray}
 A_{k}(\boldsymbol{\tau}) &=& \{\pi \in NC(k):\ \ \pi = \cup_{i=1}^{m}\pi_i,\ \pi_i \in NC(J_i (\boldsymbol{\tau}))\}, 
  \nonumber \\
 A_{t_1,t_2,\ldots, t_m, k}(\boldsymbol{\tau}) &=& \{\pi \in  A_{k}(\boldsymbol{\tau}):  \text{$\pi_i$ has $t_i+1$ blocks} \}. \nonumber
\end{eqnarray}
Note that 
\begin{eqnarray}
\cup_{t_1=0}^{T_1-1}\cup_{t_2=0}^{T_2-1}\cdots\cup_{t_m=0}^{T_m-1}A_{t_1,t_2,\ldots, t_m, k}(\boldsymbol{\tau}) = A_{k}(\boldsymbol{\tau}) \subset NC(k). \nonumber 
\end{eqnarray}

\begin{lemma}\label{lem: mpfreemoment}
Let $\{g_j, 1 \leq j \leq m\}$ be  Mar\v{c}enko-Pastur variables with parameter $y \in (0,\infty)$ on some NCP $(\mathcal{A}, \varphi)$. Then they are free  if and only if for every $k \geq 
1$, $1 \leq \tau_u \leq m$ and  $1 \leq u \leq k$, we have 
\begin{eqnarray}
\varphi(g_{\tau_1}g_{\tau_2}\cdots g_{\tau_k}) 
&=& \sum_{t_1=0}^{T_1-1}\cdots\sum_{t_m=0}^{T_m-1}  (\# A_{t_1,t_2,\ldots, t_m,k} (\boldsymbol{\tau})) y^{k-\sum_{i=1}^{m}(t_i-1)}. 
\end{eqnarray}
\end{lemma}
\begin{proof}
Suppose $\{g_i\}$ are free. For $1 \leq \tau_u \leq m$ and  $1 \leq u \leq k$, we have 
\begin{eqnarray}
\varphi(g_{\tau_1}g_{\tau_2}\cdots g_{\tau_k}) &=& \sum_{\pi \in NC(k)} \kappa_{\pi}[g_{\tau_1},g_{\tau_2},\ldots, g_{\tau_k}],\ \ \text{by (\ref{eqn: m1to14.1}}) \nonumber \\
&=& \sum_{\pi \in A_{k}(\boldsymbol{\tau})} \prod_{i=1}^{m} \kappa_{\pi_i}[g_{i},g_{i},\ldots, g_{i}],\ \ \text{by definition of free independence.} \label{eqn: mp2} 
\end{eqnarray}
Now suppose $\pi_i = \cup_{j=1}^{t_i+1} V_{ij} \in NC(J_i(\boldsymbol{\tau}))$ has $(t_i+1)$ blocks $\{V_{ij}:\ 1 \leq j \leq t_i+1\}$ of  sizes $\{\nu_{ij}:\ 1 \leq j \leq t_i+1\}$  respectively and $\sum_{j=1}^{t_i+1}\nu_{ij} = T_i$ for all $1 \leq i \leq m$. Then  by (\ref{eqn: mpcum}), we have
\begin{eqnarray}
\kappa_{\pi_i}[g_{i},g_{i},\ldots, g_{i}]  &=& \prod_{j=1}^{t_i+1} \kappa(V_{ij}) [g_{i},g_{i},\ldots, g_{i}] =  \prod_{j=1}^{t_i+1}  y^{\nu_{ij}-1} = y^{T_i-t_i-1},\ \forall\ 1 \leq i \leq m. \label{eqn: mp3}
\end{eqnarray}
Hence (\ref{eqn: mp2}) and (\ref{eqn: mp3}) together imply
\begin{eqnarray}
\varphi(g_{\tau_1}g_{\tau_2}\cdots g_{\tau_k}) &=&\sum_{t_1=0}^{T_1-1}\cdots\sum_{t_m=0}^{T_m-1}  (\# A_{t_1,t_2,\ldots, t_m, k}(\boldsymbol{\tau})) \left( \prod_{i=1}^{m} y^{T_i-t_i-1}\right) \nonumber \\
&=&\sum_{t_1=0}^{T_1-1}\cdots\sum_{t_m=0}^{T_m-1}  (\# A_{t_1,t_2,\ldots, t_m, k} (\boldsymbol{\tau})) y^{k-\sum_{i=1}^{m}(t_i-1)}. 
\nonumber 
\end{eqnarray}
The converse follows by using the one-one correspondence between $\{\varphi_{\pi}\}$ and $\{\kappa_{\pi}\}$.
This completes the proof.
\end{proof}

Suppose 
$\{X^{(i)} = ((x_{jk}^{(i)} ))_{1 \leq j\leq p, 1 \leq k \leq n} :\ 1\leq i \leq m\}$ are independent 
$p \times n$ random matrices where  $\{x_{jk}^{(i)}:\ 1 \leq j\leq p,\ 1 \leq k \leq n,\ 1 \leq i \leq m\}$ are independent across $i,j$ and $k$ and each of them has mean 0 and variance 1. Let 
\begin{eqnarray}
S^{(i)} = \frac{1}{n}X^{(i)}X^{(i)*},\ \ 1 \leq i \leq m \nonumber 
\end{eqnarray}
be the respective \textit{sample covariance matrices}.  

\begin{theorem}[\cite{CC2004}] \label{thm: mpfree}
Suppose $p\to \infty$ and $p/n \to y \in (0,\infty)$.
Suppose that $\sup_{p \geq 1}\sup_{1 \leq j \leq p} \sup_{1 \leq k \leq n} E|x_{jk}^{(i)}|^r < \infty,\ \forall\ r \geq 1$ and $1 \leq i \leq m$. Then as $p \to \infty$, $\{S^{(i)}:\ 1\leq i \leq m\}$ as elements of $(\mathcal{M}_{p}(\mathbb{C}), p^{-1}\rm{ETr})$ converge to $(g_1,g_2,\ldots, g_m)$ which are free and each $g_i$ is a Mar\v{c}enko-Pastur variable with parameter $y$.
\end{theorem}

To prove Theorem \ref{thm: mpfree} by using Theorem \ref{lem: wig}, we need one more lemma. 
A partition is said to be a \textit{pair-partition} if every block is of size $2$.
Let $NC_2(2k)$ be the set of pair-partitions in $NC(2k)$.
It is well known that $NC_2(2k)$ and $NC(k)$ have the same cardinality. 
In the following lemma we set up a specific bijection. 

For integers $1\leq j_1 < j_2 < \cdots j_r\leq k$, let $$A = \{2j_1-1,2j_1, 2j_2-1,2j_2,\ldots,2j_r-1, 2j_r\}.$$
For $\pi \in NC_2(A)$, let
\begin{eqnarray}
\mathcal{S}(\pi) = \text{Number of blocks of} \ \ \pi\ \ \text{whose first element is odd}. \nonumber
\end{eqnarray}
For $\boldsymbol{\tau} = (\tau_1,\tau_2,\ldots,\tau_k) \in \{1,2,\ldots,m\}^k$,  let 
\begin{eqnarray}
  \tilde{J}_i (\boldsymbol{\tau}) &=& \cup_{j \in J_i (\boldsymbol{\tau})}\{2j-1, 2j\}, \nonumber \\
 B_{2k}(\boldsymbol{\tau}) &=& \{\pi \in NC_{2}(2k): \pi = \cup_{i=1}^{m}\pi_i,\ \pi_i \in NC_2(\tilde{J}_i(\boldsymbol{\tau})),  1 \leq i \leq m\}, \nonumber \\
 B_{t_1, \ldots, t_m, k}(\boldsymbol{\tau}) &=& \{\pi \in B_{2k}(\boldsymbol{\tau}):  \mathcal{S}(\pi_i) = t_i+1, 1 \leq i \leq m\}, \nonumber 
\end{eqnarray}
Recall $T_i$ from (\ref{eqn: Ti}) which ranges in $\{1,2,\ldots,k\}$.  Note that 
\begin{eqnarray}
\cup_{t_1=0}^{T_1-1}\cup_{t_2=0}^{T_2-1}\cdots\cup_{t_m=0}^{T_m-1} B_{t_1, \ldots, t_m, k}(\boldsymbol{\tau}) =  B_{2k}(\boldsymbol{\tau}) \subset NC_2(2k). \label{eqn: mp7}
\end{eqnarray}

\begin{lemma} \label{lem: bijection}
For all $\boldsymbol{\tau} = (\tau_1,\tau_2,\ldots,\tau_k) \in \{1,2,\ldots,m\}^k$, $\  0 \leq t_i \leq T_i-1$, $1 \leq i \leq m$ and $k,m \geq 1$, we have
$\#B_{t_1, t_2,\ldots, t_m, k}(\boldsymbol{\tau}) =\# A_{t_1, t_2,\ldots, t_m, k}(\boldsymbol{\tau})$.
\end{lemma}
\begin{proof}
It is well known that $NC_2(2k)$ and $NC(k)$ have the same cardinality. We refine this result by setting up a bijection $f: NC_2(2k) \to NC(k)$ that also serves as a bijection between their subsets, 
$B_{t_1, t_2,\ldots, t_m, k}$ and $A_{t_1, t_2,\ldots, t_m, k}$.

For $\pi\in NC_2(2k)$, collect all the \textit{odd first elements} of blocks of $\pi$ in the following set: 
$$J_{\pi}=\{2i-1: 1\leq i \leq k, (2i-1, 2j)\in  \pi\}.$$
We now give an algorithm to construct the blocks of $f(\pi)$. 

Note that $1\in J_{\pi}$ always. Let $B_1$ be the block of $f(\pi)$, to be constructed, containing $1$. We build up $B_1$ by adding elements sequentially as follows.  Suppose $(1, 2i_1)\in \pi$. Then $i_1$ is added to $B_1$. If $i_1=1$  we stop and $B_1=\{1\}$. 
If $i_1 > 1$, then $(2i_2, 2i_1-1)\in \pi$ for some $i_2$. Then add $i_2$ to $B_1$. Now consider $(2i_3, 2i_2-1) \in \pi$ and add $i_3$ to $B_1$. We continue this till we reach $2i_k=2$ for some $k$. Then $B_1=\{1=i_k < i_{k-1} \cdots <  i_1\}$ is the block containing $1$. 

Now we move to the next available  element of $J_{\pi}$, say  $2j-1$. Suppose it matches with $2j_1$. Then the construction of the next block, say  $B_2$ begins by putting $j_1$ in it. Then we repeat the method of construction given above till we obtain $2j_1,\ldots ,  2j_t=2j$ for some $t$.  Then 
$B_2=\{j_t=j < j_{t-1} < \cdots <  j_1\}$. Note that $B_2$ is disjoint from $B_1$. We continue this process till all odd first elements are exhausted. 


The above construction leads us to the following conclusion:
\vskip3pt

\noindent 
(i) For all $\pi \in NC_2(2k)$, we have a unique $f(\pi) \in NC(k)$. 
\vskip3pt


\noindent 
(ii) For all $\sigma \in NC(k)$, we have a unique $\pi \in NC_{2}(2k)$ such that $f(\pi) = \sigma$.

\vskip 3pt

\noindent
(iii) Following statements hold for each $1 \leq j \leq m$.
\vskip 3pt
\noindent (a) For each $\pi \in NC_2(\tilde{J}_j(\boldsymbol{\tau})) \cong NC_{2}(2T_j)$, we have a unique $f(\pi_j) \in NC({J}_j(\boldsymbol{\tau})) \cong NC(T_j)$.
\vskip 3pt
\noindent (b)  For each $\sigma \in NC({J}_j(\boldsymbol{\tau})) \cong NC(T_j)$, we have a unique $\pi  \in NC_2(\tilde{J}_j(\boldsymbol{\tau})) \cong NC_2(2T_j)$ such that $f(\pi) = \sigma$.
\vskip 3pt
\noindent (iv)  For all $\pi = \cup_{j=1}^{m}\pi_j \in B_{t_1,t_2,\ldots,t_m,k}(\boldsymbol{\tau}) \subset NC_{2}(2k)$ such that 
$\pi_j \in NC_2(\tilde{J}_j(\boldsymbol{\tau}))\ \forall\ 1 \leq j \leq m$, we have  $f(\pi) = \cup_{j=1}^{m}f(\pi_j)$ where  
$f(\pi_j) \in NC({J}_j(\boldsymbol{\tau}))\ \forall\ 1 \leq j \leq m$. Hence, 
 $f(\pi) \in A_{t_1,t_2,\ldots,t_m,k}(\boldsymbol{\tau})$.

\noindent 
(v) For each $\sigma = \cup_{j=1}^{m} \sigma_j \in A_{t_1,t_2,\ldots,t_m,k}(\boldsymbol{\tau}) \subset NC(k)$ with $\sigma_j \in 
NC({J}_j(\boldsymbol{\tau}))\ \forall\ 1 \leq j \leq m$, we have a unique $\pi = \cup_{j=1}^{m}\pi_j \in B_{t_1,t_2,\ldots,t_m,k}(\boldsymbol{\tau}) \subset NC_{2}(2k)$ such that $f(\pi_j) = \sigma_j\ \forall\ 1 \leq j \leq m$ and $f(\pi) =\sigma$.

\vskip3pt

The above  observations establish Lemma \ref{lem: bijection}. 
\end{proof}

\begin{proof}[Proof of Theorem \ref{thm: mpfree}] 
By Lemma \ref{lem: mpfreemoment}, it is enough to prove that
\begin{eqnarray}
\lim \frac{1}{p} E\text{Tr}(S^{(\tau_1)}\cdots S^{(\tau_k)}) =  \sum_{t_1=0}^{T_1-1}\sum_{t_2=0}^{T_2-1}\cdots\sum_{t_m=0}^{T_m-1}  (\# A_{t_1,t_2,\ldots, t_m,k} (\boldsymbol{\tau})) y^{k-\sum_{i=1}^{m}(t_i-1)}. \nonumber
\end{eqnarray}
We establish this by embedding the $S$-matrices into larger Wigner matrices and invoking Lemma \ref{lem: wig}. 
For $1\leq i \leq m$, let $\tilde{W}_{2i-1}$ and $\tilde{W}_{2i}$ be independent Wigner matrices of order $p$ and $n$ respectively, which are also independent of $\{X_i\}$, and whose elements satisfy the assumptions in Lemma \ref{lem: wig}.  Using these, we construct the following Wigner matrices of order $(p+n)$:
\begin{eqnarray}
W_j = \bigg[\begin{array}{cc}
\tilde{W}_{2j-1} & X_j \\ 
X_j^{*} & \tilde{W}_{2j}
\end{array} \bigg],\  1 \leq j \leq m. \nonumber
\end{eqnarray}
Let $I_p$ and $I_n$ be the identity matrices of order $p$ and $n$ respectively, and $0_{p \times n}$ be the $p \times n$ matrix whose all entries are zero.  Define 
\begin{eqnarray}
\bar{I} = \bigg[\begin{array}{cc}
I_{p} & 0_{p \times n} \\ 
0_{n \times p} & 0_{n \times n}
\end{array} \bigg], \ \ \ \underline{I} = \bigg[\begin{array}{cc}
0_{p \times p} & 0_{p \times n} \\ 
0_{n \times p} & I_{n}
\end{array} \bigg]. \nonumber
\end{eqnarray}
It is easy to check that  
\begin{eqnarray}
\bigg[\begin{array}{cc}
S^{(i)} & 0_{p \times n} \\ 
0_{n \times p} & 0_{n \times n}
\end{array} \bigg] = \frac{n+p}{n}\bar{I}\frac{W_i}{\sqrt{n+p}} \underline{I} \frac{W_i}{\sqrt{n+p}} \bar{I}, \ 1 \leq i \leq m. \label{eqn: mp5}
\end{eqnarray}
Thus the $S$-matrices have been embedded in Wigner matrices of higher dimension.  

Note that the right side involves independent Wigner matrices and deterministic matrices. These deterministic matrices converge jointly and satisfy the conditions of Lemma \ref{lem: wig}. It follows that 
$(\bar{I}, \underline{I}, , (n+p)^{-1/2}W_i, 1\leq i \leq m)$ as elements of $(\mathcal{M}_{n+p}(\mathbb{C}), (n+p)^{-1}\text{ETr})$ converge in distribution to $(a_0, a_1, s_i, 1\leq i \leq m)$ where $\{s_i, 1\leq i \leq m\}$  are free semi-circular and are free of $\{a_0, a_1\}$. It is also easy to see that $a_0$ and $a_1$ are self-adjoint with 
\begin{eqnarray}
\varphi(a_0^r) = \frac{y}{1+y}\ \ \text{and}\ \ \varphi(a_1^r) = \frac{1}{1+y}\ \ \forall\ r \geq 1. \label{eqn: bermom}
\end{eqnarray}

Hence, 
\begin{eqnarray}
&& \lim \frac{1}{p} \text{ETr}(S^{(\tau_1)}\cdots S^{(\tau_k)}) 
=\lim \frac{n+p}{p} \frac{1}{n+p} \text{ETr}\Big(\prod_{j=1}^{k}\frac{n+p}{n}\bar{I}\frac{W_{(\tau_j)}}{\sqrt{n+p}} \underline{I} \frac{W_{(\tau_j)}}{\sqrt{n+p}} \bar{I}\Big),\ \ \text{by (\ref{eqn: mp5})} \nonumber \\
&=& (1+y)^{k+1} y^{-1} \lim \frac{1}{n+p}\text{ETr}\Big(\prod_{j=1}^{k}\bar{I}\frac{W_{(\tau_j)}}{\sqrt{n+p}} \underline{I} \frac{W_{(\tau_j)}}{\sqrt{n+p}} \bar{I}\Big),\ \ \text{as $n^{-1}p \to y \in (0,\infty)$} \nonumber\\
 &=& (1+y)^{k+1} y^{-1}\varphi\Big(\prod_{j=1}^{k}  a_{0} s_{\tau_j} a_{1} s_{\tau_j}a_{0}\Big),\ \ \text{by Theorem \ref{lem: wig}} \label{eqn: 3} \nonumber \\
 &=&(1+y)^{k+1} y^{-1}\sum_{\pi \in B_{2k}(\boldsymbol{\tau})} \varphi_{K(\pi)}[a_0,a_1,a_0,a_1,\ldots, a_1],\ \ \text{by (\ref{eqn: mp6}) and (\ref{eqn: s6})} \nonumber \\
 &=&(1+y)^{k+1} y^{-1} \sum_{t_1=0}^{T_1-1}\sum_{t_2=0}^{T_2-1}\cdots \sum_{t_m=0}^{T_m-1} \sum_{\pi \in B_{t_1,t_2,\ldots, t_m,k}} \varphi_{K(\pi)}[a_0,a_1,a_0,a_1,\ldots, a_1],\ \ \text{by (\ref{eqn: mp7})}
 \label{eqn: mpbr1}
 \end{eqnarray}
 Recall that $B_{t_1,t_2,\ldots, t_m,k} \subset NC_2(2k)$. It is easy to see that, for $\pi \in  NC_{2}(2k)$, any block of $K(\pi)$ contains either all even elements or all odd elements. Therefore, it follows from the definition of Kreweras complement that, for all $\pi \in B_{t_1, t_2,\ldots, t_m, k}(\boldsymbol{\tau})$, 
\begin{eqnarray}
\#\text{(Blocks in $K(\pi)$ containing only even elements)} &=& \sum_{i=1}^{m} (t_i+1), \label{eqn: 5} \\
\#\text{(Blocks in $K(\pi)$ containing only odd elements)} &=& (k+1)-\sum_{i=1}^{m} (t_i+1). \nonumber
\end{eqnarray}
Hence continuing (\ref{eqn: mpbr1}), we have
\begin{eqnarray}
&& \lim \frac{1}{p} \text{ETr}(S^{(\tau_1)}\cdots S^{(\tau_k)})  \nonumber \\
&=& (1+y)^{k+1} y^{-1}\sum_{t_1=0}^{T_1-1} \cdots \sum_{t_m=0}^{T_m-1} \ \sum_{\pi \in B_{t_1,t_2,\ldots, t_m,k}} \hspace{-10pt}\big(\frac{y}{1+y}\big)^{k+1-\sum(t_i+1)} \hspace{-2pt}\big(\frac{1}{1+y}\big)^{\sum(t_i+1)},  \text{by  (\ref{eqn: bermom}) and (\ref{eqn: 5})} 
 \nonumber \\
&=&  y^{-1} \sum_{t_1=0}^{T_1-1} \cdots \sum_{t_m=0}^{T_m-1} \ \sum_{\pi \in B_{t_1,t_2,\ldots, t_m,k}} y^{k+1-
\sum_{i=1}^{m}(t_i+1)} 
= y^{-1} \sum_{t_1=0}^{T_1-1} \cdots \sum_{t_m=0}^{T_m-1}  y^{k+1-\sum_{i=1}^{m}(t_i+1)} \#B_{t_1,t_2,\ldots, t_m,k}\nonumber \\
&=&  y^{-1} \sum_{t_1=0}^{T_1-1} \cdots \sum_{t_m=0}^{T_m-1}  y^{k+1-\sum_{i=1}^{m}(t_i+1)} \#A_{t_1,t_2,\ldots, t_m,k}, \ \text{by Lemma \ref{lem: bijection}}.\nonumber 
\end{eqnarray}
This completes the proof of Theorem \ref{thm: mpfree}.
\end{proof}

\end{document}